\documentclass[12pt]{amsart}
\usepackage{amssymb,latexsym,hyperref}
\usepackage{verbatim}   
\hypersetup{citecolor=red, linkcolor=blue, colorlinks=true}

\theoremstyle{plain}
\newtheorem{theorem}{Theorem}

\newtheorem{lemma}[theorem]{Lemma}

\theoremstyle{definition}

\theoremstyle{remark}

\begin{document}

\newcommand{\lcm}{\textup{lcm}}
\newcommand{\Mod}{\negmedspace\mod}
\newcommand{\OM}{{\mathcal{OM}}}
\newcommand{\OD}{{\mathcal{OD}}}
\renewcommand{\OE}{{\mathcal{OE}}}
\newcommand{\CM}{{\mathcal{CM}}}
\newcommand{\CD}{{\mathcal{CD}}}
\newcommand{\CE}{{\mathcal{CE}}}
\newcommand{\nOMq}{\overline{\mathcal{OM}_q}}
\newcommand{\nODq}{\overline{\mathcal{OD}_q}}
\newcommand{\nOEq}{\overline{\mathcal{OE}_q}}
\newcommand{\nCMq}{\overline{\mathcal{CM}_q}}
\newcommand{\nCDq}{\overline{\mathcal{CD}_q}}
\newcommand{\nCEq}{\overline{\mathcal{CE}_q}}
\newcommand{\Ord}{{\mathcal{O}}}
\newcommand{\Set}{{\mathcal{S}}}
\hyphenation{dividing}

\title[\tiny\upshape\rmfamily Using recurrence relations to count in symmetric groups]{}
\thanks{I am grateful to the referee for suggesting various changes
to this manuscript.}
\date{Submitted: 21 December, 2000}

\begin{center}\large\sffamily\mdseries
Using recurrence relations to count\\
certain elements in symmetric groups
\end{center}

\author{{\sffamily S.\,P. Glasby}}

\begin{abstract}
We use the fact that certain cosets of the stabilizer of points are
pairwise conjugate in a symmetric group $S_n$ in order to construct 
recurrence relations for enumerating certain subsets of $S_n$.
Occasionally one can find `closed form' solutions to such recurrence
relations. For example, the probability that a random element of $S_n$
has no cycle of length divisible by $q$ is 
$\prod_{d=1}^{\lfloor n/q\rfloor} (1-\frac{1}{dq})$.
\end{abstract}

\maketitle

\section{Introduction}\label{S:intro}
Let $S_n$ denote the symmetric group of degree $n$. If
$\Sigma\subseteq S_n$, then let $\OM_q(\Sigma)$, $\OD_q(\Sigma)$,
$\OE_q(\Sigma)$ denote the number of elements in $\Sigma$ having order:
a multiple of $q$, dividing $q$, and equal to $q$,
respectively. Similarly, let $\CM_q(\Sigma)$, $\CD_q(\Sigma)$,
$\CE_q(\Sigma)$ denote the number of elements in $\Sigma$ having a
cycle (in its disjoint cycle decomposition) of length:
a multiple of $q$, dividing $q$, and equal to $q$, respectively.
It is not hard to write down recurrence relations satisfied by
$\OM_q(C_{n,k}), \dots,\CE_q(C_{n,k})$ where $C_{n,k}$ is a certain
coset of a stabilizer of $k-1$ points. Given a function $N$, denote by
$\overline{N}$ the function defined by
$\overline{N}(\Sigma)=|\Sigma|-N(\Sigma)$. We shall give a
`closed form' solution to the recurrence relation for the number,
$\nCMq(C_{n,k})$, of elements in $C_{n,k}$ having {\it no} cycles of length
divisible by $q$. 

Asymptotic properties of the order and cycle decomposition, of a random
element of the symmetric group, $S_n$, were studied
by Erd\"os and Tur\'an in a series of seven papers entitled
``On some problems in a statistical group theory''
published between 1956 and 1972. It is shown in \cite{pEpT67b} that the
distribution $X_n$ of $\log |\tau|$, where $\tau$ is a uniformly
random element of 
$S_n$, approaches (as $n\to\infty$) the normal distribution
$N(\mu,\sigma^2)$ where $\mu=\frac12\log^2 n$ and
$\sigma^2=\frac13\log^3 n$.
The expected order, $E_n$, of a uniformly  random element of $S_n$ was
shown by Goh and 
Schmutz~\cite{wG91} to satisfy $\log E_n\sim O\left(\sqrt{n/\log n}\right)$
as $n\to\infty$. This is substantially smaller than the maximal order,
$M_n$, of an element of $S_n$ as $\log M_n\sim\sqrt{n\log n}$ as
$n\rightarrow\infty$ (see \cite[p.~222]{eL09}). 

The number of $x\in S_n$ satisfying $x^q=1$ is $\OD_q(S_n)$. Wilf
\cite{hW86} showed for fixed $q$ that 
$\OD_q(S_n)/n!\sim g_q(n)$
where $g_q(n)$ is a given function of $q$ and $n$. In a similar vein,
Pavlov \cite{aP81} showed that certain random variables associated
with cycle structure are asymptotically normal when $n\rightarrow\infty$.

One can show that the probability that an element of
$S_n$ has no cycles of lengths $a_1,\dots,a_m$ is at most
$(\sum_{k=1}^m a_k^{-1})^{-1}$ (see \cite[Theorem~VI]{pEpT67a}).
To estimate $\nCMq(S_n)$, take $a_k=qk$ and $m=\lfloor n/q\rfloor$.
Comparisons with integrals show that
$q(1+\log m)^{-1}\le \left(\sum_{k=1}^m (qk)^{-1}\right)^{-1}\le
q(\log(m+1))^{-1}.
$
This is unhelpful if $1+\log m\le q$ (as probabilities are always $\le1$).

The motivation for this work arose from the
following problem in probabilistic group theory. 
Given $\varepsilon>0$ and a group $G$ isomorphic to precisely one of
the groups $G_1,G_2,\dots$, then (when possible) determine with
probability $\ge 1-\varepsilon$ whether $G$ is isomorphic, or is not
isomorphic, to $G_k$ after testing the order of $N(\varepsilon,G)$
randomly chosen elements of $G$. This problem seems most likely to be
successful if the sequence $G_1, G_2, \dots$ comprises groups that are
finite and simple (or with few composition factors). For such groups the set
of orders of elements of $G$ frequently characterizes $G$ (see,
for example, \cite{vMwS99}). The task is clearly impossible if different
groups $G_k$ and $G_\ell$ have the same proportions of elements of each
order. It follows from the `law of large numbers' \cite{pM78}
and the above result of Erd\"os and Tur\'an that this task is
possible if $G$ is a symmetric group and $G_k=S_k$ for all $k$. By using
additional information we can give a smaller value of $N(\varepsilon,G)$.
Thus it is important to be able to quickly calculate {\it actual}
values of $\OM_q(S_n)$, $\OD_q(S_n)$, etc and not merely {\it asymptotic}
approximations as $n\rightarrow\infty$.

If $q$ is a prime-power, then \cite[Lemma~I]{pEpT67a} can be
interpreted as giving a formula for the number $\nOMq(S_n)$ of elements
of $S_n$ whose order is not a multiple of $q$. Note that
$\nOMq(\Sigma)\le\nCMq(\Sigma)$, and equality holds if $q$ is a
prime-power. We shall give a more general formula in the next section
for $\nCMq(C_{n,k})$ which specializes when $k=1$ to
$\nCMq(S_n)=\prod_{j=1}^n (j-[q\mid j])$ where $[q\mid j]$ equals~1 if $q$
divides~$j$, and~0 otherwise. If $P$ is a logical
proposition, then $[P]$ denotes 1 if $P$ is true, and 0 otherwise. This
notation, attributed to Iverson \cite[p.~24]{GKP94}, is useful for
reducing a collection of formulas involving different cases, to one formula.

\section{Recurrence relations}\label{S:Recurrence}

The symmetric group $S_n$ acts naturally on the set $\{1,\dots,n\}$.
If $k\in\{0,\dots,n\}$, let $G_{n,k}$ denote the subgroup of $S_n$
that fixes each of $1,2,\dots,k$. 
If $k\in\{1,\dots,n\}$, let $C_{n,k}$
denote the coset $G_{n,k-1}(1,2,\dots,k)$.
Note that $G_{n,k}$ is permutationally isomorphic to
$S_{n-k}$. Furthermore $C_{n,1}=S_n$ and $C_{n,n}=\{(1,2,\dots,n)\}$.

The six recurrence relations below use the ordering: 
$(n',k')<(n,k)$ if and only if $n'<n$, or $n'=n$ and $k'>k$.

\begin{lemma}\label{L:Recurrences}
Let $q,n,k$ be positive integers where $k\le n$.
Let $q=q_1\cdots q_r$ where $q_1,\dots,q_r$ are powers of distinct
primes. Let $\Delta(q,k)=\prod_{j=1}^r 
q_j^{[q_j\nmid k]}$ and $\triangledown(q,k)=\prod_{j=1}^r q_j^{[q_j\mid k]}$.
If $k<n$, then
\begin{align}
\nOMq(C_{n,k})&=\overline{\OM_{\Delta(q,k)}}(C_{n-k,1})+
  (n-k)\nOMq(C_{n,k+1}),\\
\OD_q(C_{n,k})&=[k\mid q]\OD_q(C_{n-k,1})+(n-k)\OD_q(C_{n,k+1}),\\
\nOEq(C_{n,k})&=\sum_{d\,\mid\,\triangledown(q,k)}
  \overline{\OE_{d\Delta(q,k)}}(C_{n-k,1})+(n-k)\nOEq(C_{n,k+1}),\\
\nCMq(C_{n,k})&=[q\nmid k]\nCMq(C_{n-k,1})+(n-k)\nCMq(C_{n,k+1}),\\
\nCDq(C_{n,k})&=[k\nmid q]\nCDq(C_{n-k,1})+(n-k)\nCDq(C_{n,k+1}),\\
\nCEq(C_{n,k})&=[q\ne   k]\nCEq(C_{n-k,1})+(n-k)\nCEq(C_{n,k+1}),
\end{align}
where the respective initial conditions are:
\begin{align*}
&\nOMq(C_{n,n})=[q\nmid n],&\quad
&\OD_q(C_{n,n})=[n\mid q],&\quad
&\nOEq(C_{n,n})=[q\ne n],\\
&\nCMq(C_{n,n})=[q\nmid n],&\quad
&\nCDq(C_{n,n})=[n\nmid q],&\quad
&\nCEq(C_{n,n})=[q\ne n].
\end{align*}
\end{lemma}

\begin{proof}
The initial conditions are easily verified.
Suppose now that $k<n$, and consider the coset decomposition
\[
G_{n,k-1}=G_{n,k}\cup \bigcup_{\ell\,>k} G_{n,k}(k,\ell).
\]
Post-multiplying by $(1,\dots,k)$ gives
\begin{equation}
C_{n,k}=G_{n,k}(1,2,\dots,k)\cup 
\bigcup_{\ell\,>k} G_{n,k}(1,2,\dots,k,\ell).\label{E:Lagrange}
\end{equation}
Note that the set of elements moved by $a\in G_{n,k}$ (i.e. the
support of~$a$) is disjoint from the support of $b=(1,\dots,k)$. Also,
if $\ell\,>k$, then $G_{n,k}(1,2,\dots,k,\ell)$ is the conjugate
of $C_{n,k+1}$ by $(k+1,\ell)$. Now $ab$ has no cycle of length a
multiple of $q$ if and only if $k$ is not a multiple of $q$, and $a$
has no cycle of length a multiple of $q$. That is,
\[
\nCMq(G_{n,k}(1,2,\dots,k))=[q\nmid k]\nCMq(G_{n,k})
=[q\nmid k]\nCMq(C_{n-k,1}).
\]
It follows from Eqn. (\ref{E:Lagrange}) that
\[
\nCMq(C_{n,k})=[q\nmid k]\nCMq(C_{n-k,1})+(n-k)\nCMq(C_{n,k+1}).
\]
The recurrence relations (5) and (6) are
derived similarly, and the recurrence relations (1)--(3) can be
easily derived from the facts below. Note that the order of $ab$
satisfies $|ab|=\text{lcm}(|a|,|b|)$. Hence (1) $q\nmid |ab|$ if and only
if $\Delta(q,k)\nmid |a|$; (2) $|ab|\mid q$ if and only if $|a|\mid q$ and
$k\mid q$; and (3) $|ab|= q$ if and only if $|a|=d\Delta(q,k)$ where $d\mid
\triangledown(q,k)$. 
\end{proof}

The recurrence relations for the complementary numbers
$\OM_q(C_{n,k})$, $\nODq(C_{n,k})$ etc can be determined from
those above using the fact that
$\overline{N}(\Sigma)=|\Sigma|-N(\Sigma)$.
We shall give a surprising `closed form' solution to the recurrence
relation for $\nCMq(C_{n,k})$.
Let $n~\text{mod}~q$ be the unique integer $r$ satisfying $n\equiv r\pmod q$
and $0\le r<n$. The `mod' function is notorious for not preserving
order, so the formula for $\nCMq(C_{n,k})$ below is curious as it
involves both `$\le$' and `mod'.

\begin{theorem}\label{T:Main}
If $q,n,k$ are positive integers and $1\le k\le n$, then 
\begin{equation}
\nCMq(C_{n,k})=f_q(n-k+1)-[(-k) \Mod q \le s]f_q(n-k)\label{E:main}
\end{equation}
where $f_q(n)=\prod_{j=1}^n (j-[q\mid j])$ and $s=q-2-(n \Mod q)$.
In~particular, $\nCMq(S_n)=f_q(n)$.
\end{theorem}

\begin{proof}
The result is trivially true when $q=1$. Assume henceforth that $q>1$.
We use induction on $(n,k)$ ordered  via
$(n',k')<(n,k)$ when $n'<n$, or
$n'=n$ and $k'>k$. Consider formula (\ref{E:main}) when $k=n$.
By Lemma~\ref{L:Recurrences}, $\nCMq(C_{n,n})=[q\nmid n]$. Since
\[(-n)\Mod q=[q\nmid n]q-(n\Mod q),\]
it follows that
\[
[(-n)\Mod q\le q-2-(n\Mod q)]=[\;[q\nmid n]q\le q-2]=[q\mid n].
\]
The right-hand side of (\ref{E:main}) is
$f_q(1)-[q\mid n]f_q(0)=1-[q\mid n]=[q\nmid n]$, and so (\ref{E:main}) is true
when $k=n$. Assume now that
(\ref{E:main}) is true for $(n',k')<(n,k)$ where $1\le k<n$. Thus
\[\nCMq(S_{n-k})=\nCMq(C_{n-k,1})=f_q(n-k)
\]
as $[q-1\le q-2-(n\Mod q)]=0$.

Observe that $[(-k) \Mod q \le s]=\sum_{j=0}^s [q\mid k+j]$ where
$[q\mid(k+j)]$ is abbreviated $[q\mid k+j]$. When $n\Mod
q=q-1$, then $s=-1$ and both sides are zero. (A sum $\sum_{j=0}^{-1} a_j$
is zero by convention.) Suppose that $n\Mod q<q-1$. Then
at most one summand $[q\mid k+j]$ is non-zero, and the equation
$[q\mid k+j]=1$ is equivalent to the equation
$(-k) \Mod q=j$. Hence $\sum_{j=0}^s [q\mid k+j]=[(-k) \Mod q\le s]$, as required.

We shall now prove that
\[
\nCMq(C_{n,k})=f_q(n-k+1)-f_q(n-k)\sum_{j=0}^s\;[q\mid k+j].
\]
We shorten this equation to
$\nCMq(C_{n,k})=F_{k-1}-F_k\sum_{j=0}^s[q\mid k+j]$.
The first equality below is justified by Eqn.~(4),
and the second follows from the inductive hypothesis:
\begin{align*}
\nCMq(C_{n,k})\kern-1.2pt&=[q\nmid k]\nCMq(C_{n-k,1})+(n-k)\nCMq(C_{n,k+1})\\
&=[q\nmid k]F_k
 +(n-k)\Bigl\{F_k-F_{k+1}\sum_{j=0}^s\; [q\mid k+1+j] \;\Bigr\}\\
&=(n-k+1)F_k-[q\mid k]F_k-(n-k)F_{k+1}\sum_{j=0}^s\;[q\mid
k+1+j]\\
&=(n-k+1)F_k-[q\mid k]F_k-(n-k)F_{k+1}\sum_{j=1}^{s+1}\;[q\mid
k+j].\\
\end{align*}
The last step involved a change in summation variable.

The equation
$ [q\mid n-k]\sum_{j=1}^{s+1} [q\mid k+j]=0$
is helpful. This is clearly true when $[q\mid n-k]=0$.
If $[q\mid n-k]=1$, then $k\equiv n\pmod q$ and so 
$[q\mid k+j]=[j=s+2]$. Thus in either case the expression is zero.
Using the equation $F_k=(n-k-[q\mid n-k])F_{k+1}$, therefore gives
\begin{align*}
\nCMq(C_{n,k})&=(n-k+1)F_k-[q\mid k]F_k-F_k\sum_{j=1}^{s+1}\;[q\mid k+j]\\
&=(n-k+1)F_k-F_k\sum_{j=0}^{s+1}\;[q\mid k+j]\\
&=^*(n-k+1-[q\mid n-k+1])F_k-F_k\sum_{j=0}^s\;[q\mid k+j]\\
&=F_{k-1}-F_k\sum_{j=0}^s\;[q\mid k+j]\\
\end{align*}
$\phantom{}^*$where the second last equality uses 
$[q\mid n-k+1]=[q\mid k+s+1]$ since $s\equiv -n-2\pmod q$. This
completes the inductive proof.
\end{proof}

\section{Estimations and applications}

The recurrence relations of Lemma~\ref{L:Recurrences}
give algorithms which are quadratic in $n$ for computing these
numbers. As the conjugacy classes of $S_n$ correspond bijectively to
partitions of $n$, these numbers can be computed by summing over
certain partitions. This gives rise to slower algorithms for computing
these numbers. In practice, however, we need not compute all the
significant digits of these numbers, usually the first four
suffice. Good lower bounds may be found quickly by considering some of
the large relevant conjugacy classes. 

It is a simple (and somewhat surprising) consequence of
Theorem~\ref{T:Main} that the proportion, $p_{q,n}$,
of elements of $S_n$ having no cycles of length divisible by $q$
is the same for $n=mq, mq+1, \dots, mq+q-1$. Estimates for
$p_{q,n}=\nCMq(S_n)/n!$ are obtained below.

If $q=1$, then $p_{q,n}=0$. Assume henceforth that $q\ge 2$.
Useful upper and lower bounds for
$p_{q,mq}=\prod_{k=1}^m(1-(qk)^{-1})$ may be deduced from 
\begin{align*}
\left\vert\sum_{k=1}^m\log(1-(qk)^{-1})+\sum_{k=1}^m (qk)^{-1}\right\vert
&\le \sum_{k=1}^m \left\vert\log(1-(qk)^{-1})+(qk)^{-1}\right\vert\\ 
&\le \sum_{k=1}^m\sum_{i=2}^\infty \frac{(qk)^{-i}}{2}\le \sum_{k=1}^m 
(qk)^{-2}<C 
\end{align*}
where $C=\sum_{k=1}^\infty (qk)^{-2}=q^{-2}\pi^2/6$. It follows from
\[-\frac1q(1+\log m)-C\le \sum_{k=1}^m\log(1-(qk)^{-1})\le
-\frac1q \log m+C
\]
that $c_q^{-1}(em)^{-1/q}\le \prod_{k=1}^m (1-(qk)^{-1})\le
c_qm^{-1/q}$ where $c_q=e^{q^{-2}\pi^{2}/6}$. 

Recall that one motivation for computing the numbers $\nOMq(S_n)$ etc
arose from probabilistic computational group theory. Suppose we are
given a `black box' group $G$ which is known to be isomorphic to
$S_n$ for some $n$. How do we find $n$? The relative frequency of
finding an element of $G$ of odd order should be close to the
probability $\overline{\OM_2}(S_k)/k!$ for precisely two values of
$k$, say $m$ and $m+1$. If $p$ is the smallest prime divisor of $m$ or
$m+1$, then by determining the relative frequency of elements of $G$ of
order co-prime to $p$, one can determine, with quantifiable
probability, whether $n$ equals $m$ of $m+1$.

\vskip3mm
\goodbreak
{\tiny\scshape
\begin{tabbing}
\=Central Washington University\=\kill
\>S.\,P. Glasby                \\
\>Department of Mathematics    \\
\>Central Washington University\\
\>WA 98926-7424, USA           \\
\>\scriptsize\upshape\ttfamily GlasbyS@cwu.edu\\
\end{tabbing}
}

\end{document}